\documentclass[11pt,oneside,reqno]{article}

\usepackage[a4paper,width=15cm,height=24cm]{geometry}

\usepackage{amsmath,amsthm,amssymb}
\usepackage{bbm}
\usepackage{mathrsfs}
\usepackage{graphicx}

\usepackage[authoryear,longnamesfirst]{natbib}
\usepackage{array}

\usepackage{filemod}
\usepackage[breaklinks=true]{hyperref}

\usepackage{color}

\theoremstyle{plain}
\newtheorem{theorem}{Theorem}[section]
\newtheorem{proposition}[theorem]{Proposition}
\newtheorem{lemma}[theorem]{Lemma}
\newtheorem{corollary}[theorem]{Corollary}

\theoremstyle{definition}

\newtheorem{definition}[theorem]{Definition}

\newtheorem{remark}[theorem]{Remark}

\makeatletter

\def\be#1{\begin{equation*}#1\end{equation*}}
\def\ben#1{\begin{equation}#1\end{equation}}
\def\bes#1{\begin{equation*}\begin{split}#1\end{split}\end{equation*}}
\def\besn#1{\begin{equation}\begin{split}#1\end{split}\end{equation}}

\def\ba#1{\begin{align*}#1\end{align*}}
\def\ban#1{\begin{align}#1\end{align}}

\def\given{\typeout{Command 'given' should only be used within bracket command}}
\newcounter{@bracketlevel}
\def\@bracketfactory#1#2#3#4#5#6{
\expandafter\def\csname#1\endcsname##1{%
\addtocounter{@bracketlevel}{1}%
\global\expandafter\let\csname @middummy\alph{@bracketlevel}\endcsname\given%
\global\def\given{\mskip#5\csname#4\endcsname\vert\mskip#6}\csname#4l\endcsname#2##1\csname#4r\endcsname#3%
\global\expandafter\let\expandafter\given\csname @middummy\alph{@bracketlevel}\endcsname
\addtocounter{@bracketlevel}{-1}}%
}
\def\bracketfactory#1#2#3{%
\@bracketfactory{#1}{#2}{#3}{relax}{1mu plus 0.25mu minus 0.25mu}{0.6mu plus 0.15mu minus 0.15mu}
\@bracketfactory{b#1}{#2}{#3}{big}{1mu plus 0.25mu minus 0.25mu}{0.6mu plus 0.15mu minus 0.15mu}
\@bracketfactory{bb#1}{#2}{#3}{Big}{2.4mu plus 0.8mu minus 0.8mu}{1.8mu plus 0.6mu minus 0.6mu}
\@bracketfactory{bbb#1}{#2}{#3}{bigg}{3.2mu plus 1mu minus 1mu}{2.4mu plus 0.75mu minus 0.75mu}
\@bracketfactory{bbbb#1}{#2}{#3}{Bigg}{4mu plus 1mu minus 1mu}{3mu plus 0.75mu minus 0.75mu}
}
\bracketfactory{clc}{\lbrace}{\rbrace}
\bracketfactory{clr}{(}{)}
\bracketfactory{cls}{[}{]}
\bracketfactory{abs}{\lvert}{\rvert}
\bracketfactory{norm}{\Vert}{\Vert}
\bracketfactory{floor}{\lfloor}{\rfloor}
\bracketfactory{ceil}{\lceil}{\rceil}
\bracketfactory{angle}{\langle}{\rangle}

\newcounter{ctr}\loop\stepcounter{ctr}\edef\X{\@Alph\c@ctr}%
	\expandafter\edef\csname s\X\endcsname{\noexpand\mathscr{\X}}
	\expandafter\edef\csname c\X\endcsname{\noexpand\mathcal{\X}}
	\expandafter\edef\csname b\X\endcsname{\noexpand\boldsymbol{\X}}
	\expandafter\edef\csname I\X\endcsname{\noexpand\mathbbm{\X}}
	\expandafter\edef\csname r\X\endcsname{\noexpand\mathrm{\X}}
\ifnum\thectr<26\repeat

\newcount\minute
\newcount\hour
\newcount\hourMins
\def\now{%
\minute=\time%
\hour=\time \divide \hour by 60%
\hourMins=\hour \multiply\hourMins by 60%
\advance\minute by -\hourMins%
\zeroPadTwo{\the\hour}:\zeroPadTwo{\the\minute}%
}
\def\zeroPadTwo#1{\ifnum #1<10 0\fi#1}

\numberwithin{equation}{section}
\allowdisplaybreaks[4]

\renewcommand\section{\@startsection {section}{1}{\z@}%
{-3.5ex \@plus -1ex \@minus -.2ex}%
{1.3ex \@plus.2ex}%
{\center\small\sc\mathversion{bold}\MakeUppercase}}

\def\subsection#1{\@startsection {subsection}{2}{0pt}%
{-3.5ex \@plus -1ex \@minus -.2ex}%
{1ex \@plus.2ex}%
{\bf\mathversion{bold}}{#1}}

\def\subsubsection#1{\@startsection{subsubsection}{3}{0pt}%
{\medskipamount}%
{-10pt}%
{\normalsize\itshape}{\kern-2.2ex. #1.}}

\def\blfootnote{\xdef\@thefnmark{}\@footnotetext}

\makeatother

\renewcommand{\cite}{\citet}

\def\^#1{\ifmmode {\mathaccent"705E #1} \else {\accent94 #1} \fi}
\def\~#1{\ifmmode {\mathaccent"707E #1} \else {\accent"7E #1} \fi}

\edef\-#1{\noexpand\ifmmode {\noexpand\bar{#1}} \noexpand\else \-#1\noexpand\fi}
\def\>#1{\vec{#1}}

\def\atop{\@@atop}

\renewcommand{\leq}{\leqslant}
\renewcommand{\geq}{\geqslant}
\renewcommand{\phi}{\varphi}
\newcommand{\eps}{\varepsilon}

\newcommand{\eq}{\eqref}

\newcommand{\bigo}{\mathrm{O}}
\newcommand{\lito}{\mathrm{o}}

\newcommand{\Var}{\mathop{\mathrm{Var}}\nolimits}
\newcommand{\law}{\mathscr{L}}
\def\tsfrac#1#2{{\textstyle\frac{#1}{#2}}}
\newcommand{\ER}{Erd\H{o}s-R\'{e}nyi}

\newcommand{\Bi}{\mathop{\mathrm{Bi}}}

\newcommand{\eqlaw}{\stackrel{d}{=}}

\newcommand{\ergm}{\mathrm{ERGM}}
\newcommand{\glab}{\cG^{\mathrm{lab}}}
\def\gset#1{[#1]_{2}}
\def\s#1{^{(#1)}}
\def\t#1{^{[#1]}}
\newcommand{\dham}{d_{\rH}}
\newcommand{\astar}{a^*}

\begin{document}

\title{\sc\bf\large\MakeUppercase{Approximating stationary distributions   
of fast mixing Glauber dynamics, with applications to exponential random graphs}}
\author{\sc
Gesine Reinert\thanks{Oxford University; {\tt reinert@stats.ox.ac.uk}}\, and Nathan Ross\thanks{University of Melbourne; {\tt nathan.ross@unimelb.edu.au}} }
\maketitle

\begin{abstract}  
We provide a general bound on the Wasserstein distance between two arbitrary distributions of sequences of Bernoulli random variables. The bound is in terms of a mixing quantity for the Glauber dynamics of one of the sequences, and a simple expectation of the other.
The result is applied to estimate, with explicit error, expectations of functions of random vectors for some Ising models and exponential random graphs in ``high temperature" regimes.
\end{abstract}


\section{Introduction}
A high-level heuristic in statistical physics models is that 
systems of dependent random variables having the stationary distribution
of a fast mixing Markov chain are approximately independent. 
Precise statements of this heuristic 
 include spatial-temporal mixing 
conditions (see for example \cite{Mossel2013} and discussion there); 
as well as establishing properties of the dependent variables common to 
independent variables, such as central limit theorems (see for example \cite{Ellis1978});
and concentration inequalities (see for example \cite{Luczak2008}, \cite[Chapter~4]{Chatterjee2005a}, and \cite{Gheissari2017}).
Such results are particularly useful for models  with intractable normalizing constants; for example, Ising models and exponential random graphs;
since it is possible to define and analyze Markov chains 
which only depend on ratios of probabilities and therefore do not require computation of the normalizing constant.

We provide a general result (Theorem~\ref{thm:key} below) that explicitly bounds the Wasserstein distance between distributions of vectors of Bernoulli random variables, in terms of 1) a quantity depending on the mixing 
of the Glauber dynamics of one of the vectors; and 2) a straightforward expectation of the other. The result applies to many models where fast mixing is known, and here we apply it to Ising models and exponential random graphs, sharpening and generalizing some recent results of \cite{Eldan2017a, Eldan2017}.

We  now state a simplified version of our general approximation result, and its implications in some applications;~\eq{eq:cwbd} and Propositions~\ref{prop:2star} and~\ref{prop:tri}.
Let $N\in\IN$ and  $X=(X_1,\ldots, X_N)\in\{0,1\}^{N}$ be a vector of Bernoulli random variables. 
Denote $[N]:=\{1,\ldots,N\}$, and for $s\in[N]$, write
${\rm e}_{s}\in\{0,1\}^{N}$ for the vector with~$1$ in the $s$th coordinate, and 0 in all others, and, for $x\in\{0,1\}^{N}$, denote 
\be{
x\s{s,1}=x+(1-x_{s}){\rm e}_{s}, \hspace{1cm} x\s{s,0}=x-x_{s}{\rm e}_{s},
}
so that $x\s{s,1}$ has a~$1$ in the $s$th coordinate and is otherwise the same as~$x$, and similar for $x\s{s,0}$ except there is a~$0$ in the $s$th coordinate.
 
Define the Glauber dynamics Markov chain $\{ X(m), m =0, 1, \ldots \}$ with transition probabilities 
\be{
\IP\bclr{x\mapsto x\s{s,1}}=\frac{1}{N}-\IP\bclr{x\mapsto x\s{s,0}}=\frac{q_X\bclr{x\s{s,1}|x}}{N}, \quad s\in[N],
}
where
\be{
q_X\bclr{x\s{s,1}|x}:=\IP\bclr{X_s=1|(X_u)_{u\not=s}=(x_u)_{u\not=s}}.
} 
In words, at every step in the chain, we choose a coordinate uniformly at random and
resample it conditional on the values at the other coordinates.
It is easily checked that this chain is reversible with respect to $\law(X)$.

For functions $h:\{0,1\}^N\to \IR$, define
 \be{
 \Delta_{s} h(x):= h(x\s{s,1})- h(x\s{s,0}),
 }
denote the supremum norm by $\norm{\cdot}$, let $\norm{\Delta h}:=\sup_{s\in[N]} \norm{\Delta_s h}$, and   denote the Hamming distance of $x,y\in\{0,1\}^N$ by   $\dham(x,y):=\sum_{s=1}^N \abs{x_s-y_s}$.

\begin{theorem}\label{thm:key1}
Let  $X, Y\in\{0,1\}^N$ be random vectors.
Write
\ben{\label{exponentialform} 
\IP(X=x)=\frac{1}{\kappa} \exp\{L(x)\},
}
for a  function~$L$ and normalising constant~$\kappa$, and
assume  that $Y$ has independent coordinates with $p_s:=\IP(Y_s=1)$, for $s=1,\ldots, n$, satisfying 
\ben{\label{eq:fp2}
p_s=\frac{1+\tanh\bclr{\frac{1}{2} \IE \Delta_s L(Y)}}{2}.
}
Assume that for each $s\in[N]$ and $x\in\{0,1\}^N$,
 there is a coupling $(\tilde U\t{x,s}, \tilde V\t{x,s})$
with ${\mathcal{L}} (\tilde U\t{x,s}) = {\mathcal{L}}(X(1) | X(0)=x\s{s,1})$ and ${\mathcal{L}} (\tilde V\t{x,s}) =  {\mathcal{L}}(X(1) | X(0)=x\s{s,0})$,
 such that for some $0<\rho\leq 1$, 
\ben{\label{eq:glaubcon}
 \IE \, \dham(\tilde U\t{x,s}, \tilde V\t{x,s})\leq (1-\rho),
 }
 and that the Glauber dynamics of~$\law(X)$ are irreducible.
Then for any $h:\{0,1\}^N\to\IR$, 
\ben{\label{eq:1.1} 
\babs{\IE h(X)-\IE h(Y)} \leq   \frac{\norm{\Delta h}}{4N \rho}\sum_{s=1}^N \IE \babs{\Delta_s L(Y)-\IE \Delta_s L(Y)}.
}
\end{theorem}
Before applying the result, we make a couple of general remarks to aid in interpretation.
\begin{remark}
Regarding the conditions of the theorem: writing the probabilities of~$X$  in the form \eqref{exponentialform}
is without loss of generality assuming $\IP(X=x)>0$ for all $x\in \{0,1\}^N$ (or we allow $L$ to take the value $-\infty$). The condition~\eq{eq:fp2} comes from matching the update probabilities of
the Glauber dynamics of $\law(X)$ to those of the independent chain, while respecting the behaviour of the statistic in the exponent under the product measure; this is the so-called mean field prediction,
c.f., \cite{Eldan2017a}. Solutions to this equation also arise
as critical points for variational problems arising from studying large
deviations in these systems; see \cite{Ellis2006}, \cite{Chatterjee2013}, and discussion below.

Chains with a coupling satisfying~\eq{eq:glaubcon} are  called \emph{contracting} in
\cite{Gheissari2017}, and the condition is 
satisfied by many chains in statistical physics
models in ``high temperature" regimes; see \cite[Chapters~14 and~15]{Levin2009} and applications below. 
The condition is implied by the Dobrushin-Shlosman condition, going back to 
\cite{Dobrushin1985}, which is a classical criterion for fast mixing of the Glauber dynamics, among other nice implications.
Our general result, Theorem~\ref{thm:key},  replaces the Dobrushin-Shlosman condition~\eq{eq:glaubcon} by the more general assumption that the ``influence matrix" of the Glauber dynamics has $p$-norm less than~$1$ (Dobrushin-Shlosman corresponds to $p=1$). Such an assumption still implies fast mixing of the Glauber dynamics; see
 \cite{Dobrushin1970}, \cite{Hayes2006}, and \cite{Dyer2009}.

Also note that under the assumptions of the theorem, the bound is easy to compute since the expectation in~\eqref{eq:1.1} is taken against a vector~$Y$ of \emph{independent} variables. 
A key to our approach 
is that  the intractable normalising constant~$\kappa$ of~$\law(X)$ does not need to be computed in order to analyse the Glauber dynamics.
\end{remark}

\begin{remark}
In a nice situation, $\rho\asymp 1/N$, $\IE \babs{\Delta_s L(Y)-\IE \Delta_s L(Y)}=\bigo(1/\sqrt{N})$, and the bound is then $\norm{\Delta h}\bigo(\sqrt{N})$.
The way to interpret the bound is that for 
functions $h$ such that 
$\norm{\Delta h} \sqrt{N} \ll \IE h(Y)$, $\IE h(X)$ can be approximated by $\IE h(Y)$ with small error. 
If $\IE h(Y)$ is of constant order with $N$, then the function $h$ must satisfy
$\norm{\Delta h}=\lito(1/\sqrt{N})$ for the bound to be meaningful. 
\end{remark}
\begin{remark}

In the case that the chain does not satisfy the mixing assumption~\eq{eq:glaubcon}, or that of the more general Theorem~\ref{thm:key}, 
Lemmas~\ref{lem:genapp1} and~\ref{lem:delf} give an intermediate bound that can be used directly, as is done to prove Theorem~\ref{thm:hightemp} below.
\end{remark}

We apply Theorem~\ref{thm:key1} in some examples and applications.

\subsection{Ising model on a fixed graph}\label{sec:ising}

Let $G$ be a fixed graph on $N$ vertices with vertex set $V(G) = [N]$ and edge set $E(G)$. Each vertex has a label in $\{0,1\}$ attached to it, and a configuration of such labels is denoted by $x \in \{0,1\}^{N}$.
Let $r\in[N-1]$ and for each $s\in[N]$, let $\cN_s\subset [N]\setminus\{s\}$ be a non-random set of size
$\abs{\cN_s}\leq r$, such that $u \in \cN_s \iff s \in \cN_u$.  One can think of $r$ as the largest vertex degree in $G$. Let $X\in  \{0,1\}^{N} $ have the Ising model with ``neighbourhoods" $(\cN_s)_{s\in[N]}$, defined by $\IP(X=x)\propto e^{L(x)}$, where
\be{
L(x)=\frac{\beta}{N} \sum_{s=1}^N \sum_{t\in \cN_s} (2x_s-1)(2x_t-1).
}
To apply Theorem~\ref{thm:key1}, we first compute
\begin{eqnarray*}
\Delta_s L(x) &= &  \frac{\beta}{N} \left\{\sum_{t\in \cN_s} (2-1)(2x_t-1) - 
\sum_{t\in \cN_s} (0-1)(2x_t-1) \right.  \\
&& \left. +  \sum_{u: u\ne s} 
\mathbbm{1}( s \in \cN_u) \{  (2x_u-1)(2-1)- (2x_u-1)(0-1) \} 
\right\} \\
&=& 4 \frac{\beta}{N} \sum_{t\in \cN_s} (2x_t-1),
\end{eqnarray*} 
where we used that $u \in \cN_s \iff s \in \cN_u$.
Hence 
\be{
\Delta_s L(x)=\frac{4 \beta  }{N} \sum_{t\in \cN_s} (2 x_t-1). 
}
For  $(Y_s)_{s\in[N]}$ a vector of independent Bernoulli
variables with $\IP(Y_s=1)={p_s}$,
\be{
\IE \Delta_s L(Y)={\frac{4 \beta }{N} \sum_{t\in\cN_s}(2p_t-1)},
}
and according to~\eq{eq:fp2}, we set the $p_s$ 
to satisfy
\begin{eqnarray} \label{eq:CWsol}
{(2p_s-1)=\tanh\bbbclr{\frac{2\beta}{N}\sum_{t\in\cN_s}(2p_t-1)}}.
\end{eqnarray} 
Equation \eqref{eq:CWsol} always has at least one solution: $p_s \equiv a=1/2$.
For the mixing time, we restrict to the high temperature regime 
where $0<\beta<r/N$. According to \cite[Proof of Theorem~15.1]{Levin2009}, with their $\beta$ corresponding to our $\beta/N$ 
(and see also \cite[Proposition~2.1]{Levin2010} for the case $r=(N-1)$),
there is a coupling of $(U\t{x,s}(1), V\t{x,s}(1))$ such that
\be{
\IE \, \dham(U\t{x,s}(1), V\t{x,s}(1)) \leq \bclr{1-N^{-1}(1-r\tanh(\beta /N))},
}
so that in applying Theorem~\ref{thm:key1}, we can set  $\rho=N^{-1}(1-r \tanh(\beta /N))\geq (1-\beta r/N)N^{-1}$; since $\tanh$ is $1$-Lipschitz.
Hence Theorem~\ref{thm:key1} implies that 
or any $h:\{0,1\}^N\to\IR$, 
\ben{\label{eq:1.1b} 
\babs{\IE h(X)-\IE h(Y)} \leq   \frac{\norm{\Delta h}}{4(1 - \beta r / N) }\sum_{s=1}^N \IE \babs{\Delta_s L(Y)-\IE \Delta_s L(Y)}.
}

Assuming  $p_s\equiv a$, 
it is easy to see that  $\Delta_s L(Y)\eqlaw 8 \beta N^{-1} {B_s}+ c$, where ${B_s\sim\Bi(\abs{\cN_s}},a)$ and $c = - 4 \beta{\abs{\cN_s}} / N$ is a constant. So with $a = \frac12$ and $Y$ a vector of i.i.d.\ Bernoulli$(1/2)$ random variables,
\be{
\IE \abs{\Delta_s L(Y)-\IE \Delta_s L(Y)} \leq \sqrt{\Var(\Delta_s L(Y))}  = \frac{8\beta \sqrt{{\abs{\cN_s}}a(1-a)}}{N}{\leq} \frac{4\beta \sqrt{r}}{N}.
} 
Combining this and~\eq{eq:1.1b}, Theorem~\ref{thm:key1} implies that for any $h:\{0,1\}^N\to\IR$ and $0\leq \beta< r/N$,
\ben{\label{eq:cwbd}
\babs{\IE h(X)-\IE h(Y)}\leq\norm{\Delta h} \frac{\beta \sqrt{r}}{\bclr{1-\beta r/N}}.
}

To put this result in context, the model with ${\abs{\cN_s}}=(N-1)$ is frequently referred to as the Curie-Weiss model and is one of the most well-studied in statistical physics. Laws of large numbers, central  limit theorems with rates, large deviations, concentration and moment inequalities, and local limit theorems with rates have been established for this model; see \cite{Ellis1978}, \cite{Ellis1980}, \cite{Ellis2006}, \cite{Chatterjee2007}, \cite{Eichelsbacher2010},  \cite{Chatterjee2011a}, \cite{Rollin2015}, \cite{Barbour2017}. 
However, it is not obvious how to get an explicit
result like~\eq{eq:cwbd} directly from 
these results. Some very recent work 
closely related to~\eq{eq:cwbd} is \cite[Corollary~13]{Eldan2017a}, which,
under some conditions on $\beta$, gives
\ben{\label{eq:eldisbd}
\babs{\IE h(X)-\IE h(Y)}\leq  601 \norm{\Delta h} e^{2(1+\beta)} \frac{1+\beta}{1-\beta} N^{7/8}.
}
Our bound~\eq{eq:cwbd} in this case is essentially $ \norm{\Delta h}\beta \sqrt{N}/(1-\beta)$, which compares favorably to~\eq{eq:eldisbd} in rate, constant, and dependence on $\beta$ (i.e., goes to zero with $\beta$).


\subsection{Exponential random graph models}
Exponential random graph models (ERGMs), suggested for directed networks by \cite{Holland1981} and for undirected networks by \cite{Frank1986},
are frequently used as parametric statistical models in network analysis, see for example \cite{Wasserman1994}. This is due to the Gibbs form of the distribution, which allows for straightforward implementation of modern Markov chain Monte Carlo (MCMC) methods. However, the models are difficult to analyze directly, and so the stability of MCMC algorithms and the structure of the resulting networks have only recently been meaningfully studied.  Two important references, both in general and for this work, are \cite{Bhamidi2011}, which establishes mixing times for the MCMC (Gibbs samplers) used for ERGMs, and \cite{Chatterjee2013}, which establishes asymptotic properties of these random networks. A further discussion is postponed to after the statement of our definitions and results.

Of particular interest in the analysis of networks are counts of small graphs. These counts occur in the theory of graph limits, see Section 2.2 in \cite{Chatterjee2013} for an overview. Over-and under-represented small subgraphs are also conjectured to be  building blocks of complex networks, see for example \cite{Milo2002}. Small subgraphs are also used for comparing networks; see \cite{Przulj2007}, \cite{Ali2014}. Assessing the exceptionality of small subgraph counts, also called {\it motifs}, depends crucially on the underlying network model, see \cite{Picard2008}. Hence assessing the distribution of small subgraph counts in ERGMs is a key question. While our results do not provide the distribution of these counts, they provide an approximation in terms of the well-studied corresponding distributions for Bernoulli random graphs, together with a bound on the error in the approximation which allows to gauge whether the approximation is appropriate.

We first define the vertex-labeled ERGM.
Let $\glab_n$ be the set of vertex-labeled simple graphs on~$n$ vertices and 
define the set $\gset{n}:=\{(i,j):1\leq i <j \leq n\}$ (this is sometimes also denoted $\binom{[n]}{2}$ but we prefer our notation for typesetting purposes). We 
identify $\glab_n$ with $\{0,1\}^{\binom{n}{2}}$ by encoding
$x\in\glab_n$ by an ordered collection of $0-1$ valued variables: $x=(x_{ij})_{(i,j)\in\gset{n}}$, where $x_{ij}=1$ means there is an edge between 
vertices~$i$ and~$j$. We refer to $x_{ij}$ as the $ij$th ``coordinate" 
of~$x$. 
Note that in the general setup above, $N=\binom{n}{2}$.

For a graph $H$, let $V(H)$ denote the vertex set, 
and for $x\in\{0,1\}^{\binom{n}{2}}$, define $t(H,x)$ to be the number of
``edge-preserving" injections from $V(H)$ to $V(x)$; an injection $\sigma$ preserves edges if for all edges $vw$ of $H$, $x_{\sigma(v)\sigma(w)}=1$ (here assuming $\sigma(v)<\sigma(w)$). 

\begin{definition}[Exponential random graph model]\label{def:ergm}
Fix $n\in\IN$ and $k$ connected  graphs $H_1,\ldots,H_k$ with $H_1$ a single edge, and for $\ell=1,\ldots,k$ denote $v_\ell:=\abs{V(H_\ell)}$ (so $v_1=2$), and 
\be{
t_\ell(x)=\frac{t(H_\ell,x)}{n(n-1)\cdots(n-v_\ell+3)}.
}
For $\beta=(\beta_1,\ldots,\beta_k)$ with $\beta_\ell\in\IR$ 
 for $\ell=1,\ldots, k$, 
we say the random graph $X\in\glab_n$ is distributed according to the exponential random graph model with parameters $\beta$, denoted $X\sim\ergm(\beta)$, if for $x\in \glab_n$,
\be{
\IP(X=x)=\frac{1}{\kappa_{n}(\beta)}\exp\bbbbclr{\sum_{\ell=1}^k \beta_\ell t_\ell(x)},
}
where $\kappa_n(\beta)$ is a normalizing constant.
\end{definition}
The scaling in the exponent matches \cite[Definition~1]{Bhamidi2011} and \cite[Sections~3 and~4]{Chatterjee2013}. Note also that $t_1(H_1,x)$ is twice the number of edges of~$x$ so that 
if $k=1$, then $\ergm(\beta)$ has the same law as an \ER\ random
graph (identified above as a collection of $\binom{n}{2}$ i.i.d.\ Bernoulli variables) with edge parameter~$e^{2\beta_1}/(1+e^{2\beta_1})$.

As in \cite{Bhamidi2011} and \cite{Eldan2017}, an important role is
played by the following functions defined on $[0,1]$,
\be{
\Phi(a):=\sum_{\ell=1}^k \beta_\ell \, e_\ell \, a^{e_\ell-1}, \hspace{3mm} \mbox{ and } \hspace{2mm} 
\phi(a):=\frac{1+\tanh(\Phi(a))}{2}=\frac{e^{2\Phi(a)}}{e^{2\Phi(a)}+1},
}
where $e_\ell$ is the number of edges in $H_\ell$. 
In particular, solutions to the equation $\phi(a)=a$ satisfying $\phi'(a)<1$
are key quantities in what follows. 
To help understand where these equations come from, \cite[Theorem~4.2]{Chatterjee2013}, states that if $\beta_\ell\geq0$ for $\ell=2,\ldots,k$, then $X\sim\ergm(\beta)$ is asymptotically close (in the cut metric) to a mixture of \ER\ random graphs where the mixture is over the finite set~$U$ of maximizers in the interval $[0,1]$, of the function 
\ben{\label{eq:posbvar}
\sum_{\ell=1}^k \beta_\ell a^{e_\ell}-\frac{1}{2}\bclr{a \log(a)+(1-a)\log(1-a)}.
}
Moreover, \cite{Chatterjee2013} show that the set $U$ is finite, and possibly only has one element. The connection to solutions of $\phi(a)=a$ is that critical points of~\eq{eq:posbvar} satisfy
\ben{\label{eq:posbvarcrit}
2\Phi(a)=\log\bbbclr{\frac{a}{1-a}},
}
which a little algebra shows is the same as $\phi(a)=a$. 
The second condition that $\phi'(a)<1$ corresponds to the critical point being a
local maximum. Thus, such solutions are key to describing any \ER\ limiting behaviour of ERGMs.

It is intuitively clear that if $\beta_2,\ldots,\beta_k$ are small, then $\ergm(\beta)$ should be close to an \ER\ random graph.  Our first $\ergm$ result explicitly quantifies this heuristic.
Here and below, we denote 
\be{
C_2:=\frac{4}{3\sqrt{3}}< 0.77,
}
which appears in our study as the maximum of the first derivative of $\mathrm{sech}^2(a)$. 
{Moreover, from the next theorem onwards, $Z$ always has  the~\ER\ distribution with parameter~$\astar$, where $\astar\in[0,1]$ satisfies $\astar=\phi(\astar)$.}

\begin{theorem}\label{thm:smallbetas}
For given $\beta_1\in\IR$, $\beta_\ell>0$, $\ell=2,\ldots,k$, assume $\astar\in[0,1]$ satisfies $\astar=\phi(\astar)$, define $A^*:=\max\{\astar,1-\astar\}\leq 1$, set
\bes{
\alpha_1&:=\frac{1}{2}\bclr{\Phi'(\astar)+A^*\Phi''(1)}, \\
\alpha_2&:=\phi'(\astar) +\frac{1}{2}\bbcls{  C_2\left(A^* +n^{-1}\right) \Phi'(1) (\Phi'(\astar)+A^* \Phi'(1))	+ A^* \Phi''(1)\mathrm{sech}^2\bclr{\Phi(\astar)}},	
}
and assume
\ben{\label{eq:gdcond}
1-\min\{\alpha_1, \alpha_2\} =:\gamma>0.
}
Then for $X\sim\ergm(\beta)$, $Z$ having the~\ER\ distribution with parameter~$\astar$, 
and $h:\{0,1\}^{\binom{n}{2}}\to \IR$,
\be{
\abs{\IE h(X)-\IE h(Z)}
	\leq \norm{\Delta h}\binom{n}{2}(4 \gamma)^{-1} \sum_{\ell=2}^k \beta_\ell \sqrt{\Var(\Delta_{12} t_\ell(Z))}.
}
\end{theorem}

The condition~\eq{eq:gdcond} may look contrived, but note the theoretically nice fact that a necessary condition for $1-\alpha_2>0$, is that $\phi'(\astar)<1$, and this corresponds to the critical point $\astar$ of~\eq{eq:posbvar} being a local maximum, as explained above.
We now prove a similar result, which has simpler conditions and allows
$\beta_\ell$  to be negative for $\ell\geq2$. Showing how our methods can be used to extend to this case is due to \cite[Proofs of Theorem~2.2 and Proposition~2.4]{Sinulis2018}, which appeared after the first draft of this paper was put on the arXiv.
Define $\abs{\Phi}(a):=\sum_{\ell=1}^k \abs{\beta_\ell} \, e_\ell \, a^{e_\ell-1}$.

\begin{theorem}\label{thm:negbetas}
For given $\beta_\ell\in\IR$, $\ell=1,\ldots,k$, assume that $\frac{1}{2}\abs{\Phi}'(1)<1$ and that 
$\astar\in[0,1]$ satisfies $\astar=\phi(\astar)$.
Then for $X\sim\ergm(\beta)$, $Z$ having the~\ER\ distribution with parameter~$\astar$, 
and $h:\{0,1\}^{\binom{n}{2}}\to \IR$,
\be{
\abs{\IE h(X)-\IE h(Z)}
	\leq \norm{\Delta h}\binom{n}{2}\bbclr{4\bclr{1-\tsfrac{1}{2}\abs{\Phi}'(1)}}^{-1} \sum_{\ell=2}^k \abs{\beta_\ell} \sqrt{\Var(\Delta_{12} t_\ell(Z))}.
}
\end{theorem}

\begin{remark}
The condition from Theorem~\ref{thm:negbetas},
\ben{\label{eq:simpcond}
\abs{\Phi}'(1)<2,
}
is easier to verify than~\eq{eq:gdcond} of Theorem~\ref{thm:smallbetas}.
Moreover, the dependence on $\beta$ of~\eq{eq:simpcond} is more transparent than~\eq{eq:gdcond}, since it does not  involve the fixed point~$\astar$.
In particular, it is  plain to see that for any fixed choice of $H_2,\ldots,H_k$,~\eq{eq:simpcond} holds for $\abs{\beta}$ small enough.
\end{remark}
\begin{remark}\label{rem:var}
For fixed $\astar$, the random variables 
\be{
\Delta_{(1,2)} t_\ell(Z)=\frac{\Delta_{(1,2)} t(H_\ell, Z)}{n(n-1)\cdots(n-v_\ell+3)},
}
have variance of order at most~$1/n$, with constant depending on properties of~$H_\ell$. 

To see this, let~$H$ and~$x$ be graphs with vertex and edge sets $V(H), V(x), E(H), E(x)$, and let
${I}(H,x)$  be the set of all injections $i_{H, x}: V(H) \rightarrow V(x)$. For such an injection and any edge ${e}=(u,v)$ {of~$H$}, we use the notation $i_{H,x}({e}) = {\{}i_{H,x}(u), i_{H,x}(v){\}}$ and $E(i_{H,x}) = \cup_{{e} \in E(H)} \{i_{H,x}({e})\}$. Then the number of edge-preserving injections of ${I}(H,x)$ is 
\be{
t(H,x) =\sum_{i_{H,x} \in {I}(H,x)} \prod_{e \in E( H)} {\mathbbm 1} \bclr{i_{H,x}(e) \in E(x)},
}
and so for $s \in E({x})$, 
\bes{
\Delta_s t(H,x)
&=\sum_{i_{H,x} \in {I}(H,x)} \left\{ \prod_{e \in  E( H)}  {\mathbbm 1}  \bclr{i_{H,x}(e) \in E(x\s{s,1} )}
- \prod_{e \in H}  {\mathbbm 1}  \bclr{i_{H,x}(e) \in E(x\s{s,0} )} \right\} \\
&= \sum_{i_{H,x} \in I(H,x)}   {\mathbbm 1} \bclr{s \in E(i_{H,x}) } \prod_{e \in  E( H) \setminus {{i_{H,x}^{-1}(s)}} } {\mathbbm 1}  \bbclr{i_{H,x}(e) \in \bclr{ E(x) \setminus i_{H,x}(s)}}.
} 
Now, ${\abs{I(H,x)}}$ is of the order $n^{|v(H)|} $ and the number of such edge-preserving injections~$i_{H,x}$ which use the edge $s$, so that $s \in E(i_{H,x})$, is of the order $n^{|v(H)|-2} $. The variance of $ \Delta_s t(H,Z)$ involves of the order $n^{2(|v(H)|-2)} $ covariances, but many of these covariances will be 0 because edges in $Z$ are independent. Indeed, a non-zero covariance is obtained only when the two injections share at least one edge, which leaves only at most $|v(H)|-3$ vertices to choose from for the injection. Hence 
$$\Var \Delta_s t(H,Z) = O(n^{2 |v(H)|-5}).$$
Thus  taking the denominator in $ \Delta_{(1,2)} t_\ell(Z) $ into account yields $ \Var \Delta_{(1,2)} t_\ell(Z) = O(n^{-1})$. This argument has been explored for subgraph counts in Bernoulli graphs in many previous works, for example \cite{Rucinski1988}.
\end{remark}

Special test functions of interest are homomorphism densities
 $h(x)=t(H,x) n^{-\abs{V(H)}}$ for 
small nontrivial subgraphs $H$, since 
 convergence of subgraph 
 densities plays a key role in many other notions of convergence; see \cite{Borgs2008, Borgs2012},
 and, as mentioned previously, counts of small subgraphs are
 used in the analysis of networks as building blocks and summary statistics.
We have the following corollary, 
which follows easily from Theorems~\ref{thm:smallbetas} and~\ref{thm:negbetas} and Remark~\ref{rem:var} after noting that $\norm{\Delta h}=\bigo(n^{-2})$,  providing a rate of convergence for this case.

\begin{corollary}\label{cor:vhightemp}
If $\beta$ satisfies the hypotheses of Theorem~\ref{thm:smallbetas} or~\ref{thm:negbetas}, and for some graph~$H$, 
 $h(x)=t(H,x) n^{-\abs{V(H)}}$,  then
for $X\sim\ergm(\beta)$, $Z$ having the~\ER\ distribution with parameter~$\astar$,
\be{
\abs{\IE h(X)-\IE h(Z)}\leq \frac{c}{\sqrt{n}},
}
where $c$ is a computable constant depending on $\beta$ and $H$.
\end{corollary}

Explicit computation of the region of applicability of Theorem~\ref{thm:negbetas} is possible for small -- though frequently used (see \cite{Wasserman1994}) -- examples, as we now illustrate.

\begin{proposition}\label{prop:2star}
Let $X\sim\ergm(\beta)$ with $k=2$, $H_2$ a two-star, $\beta_1\in\IR$, and $\abs{\beta_2}<1$.
Then there is a unique $\astar\in[0,1]$ satisfying $\astar=\phi(\astar)$,
and for~$Z$ distributed \ER\ with parameter $\astar$, and any $h:\{0,1\}^{\binom{n}{2}}\to \IR$,
\be{
\abs{\IE h(X)-\IE h(Z)}
	\leq \norm{\Delta h}\binom{n}{2} \bbclr{4\bclr{1-\abs{\beta_2}}}^{-1}\,\frac{ \sqrt{8\astar(1-\astar)}\abs{ \beta_2}}{\sqrt{n-2}}.
}
\end{proposition}

\begin{proposition}\label{prop:tri}
Let $X\sim\ergm(\beta)$ with $k=2$, $H_2$ be a triangle, $\beta_1\in\IR$, and $\abs{\beta_2}<1/3$.
Then there is a unique $\astar\in[0,1]$ satisfying $\astar=\phi(\astar)$,
and for~$Z$ distributed \ER\ with parameter $\astar$ and any $h:\{0,1\}^{\binom{n}{2}}\to \IR$,
\be{
\abs{\IE h(X)-\IE h(Z)}
	\leq \norm{\Delta h}\binom{n}{2} 
	\bbclr{4\bclr{1-3\abs{\beta_2}}}^{-1} \, \frac{ 6\astar\sqrt{1-(\astar)^2} \abs{\beta_2}}{\sqrt{n-2}},
}
\end{proposition}

For comparison, if $k=2$ and $H_2$ is a triangle, the recent result of \cite[Theorem~19]{Eldan2017} states that in the range $0\leq \abs{\beta_2}< 1/54$,
\be{
\abs{\IE h(X)-\IE h(Z)}
	\leq \norm{\Delta h}\frac{600 \exp\bclc{12 (\abs{\beta_1}+9\abs{\beta_2}) \vee 2}}{1-54\beta_2}\binom{n}{2}^{15/16}.
}
Our bound compares favourably to this one in rate, constant, and dependence on and range of $\beta_2$.

Before stating our next $\ergm$ result and having a further discussion, we give a quick toy numerical example. 

\smallskip

\noindent{\bf Florentine marriages.}
The Florentine Families Marriage data is a classical data set collected by \cite{Padgett1993}. 
It is an undirected network, which consists of the marriage ties among 16 families in 15th century Florence, Italy; each family is a vertex, and two vertices are connected by an edge if there is a marriage between them. There are $n=16$ vertices and $20$ edges in the network, and one of the vertices is isolated; the network has $50*2=100 $ subgraphs which are isomorphic to a 2-star. Figure \ref{florentine} shows this network; the vertices are labelled by the family names.

\begin{figure}[h!] 
\centering
\includegraphics[width=13cm]{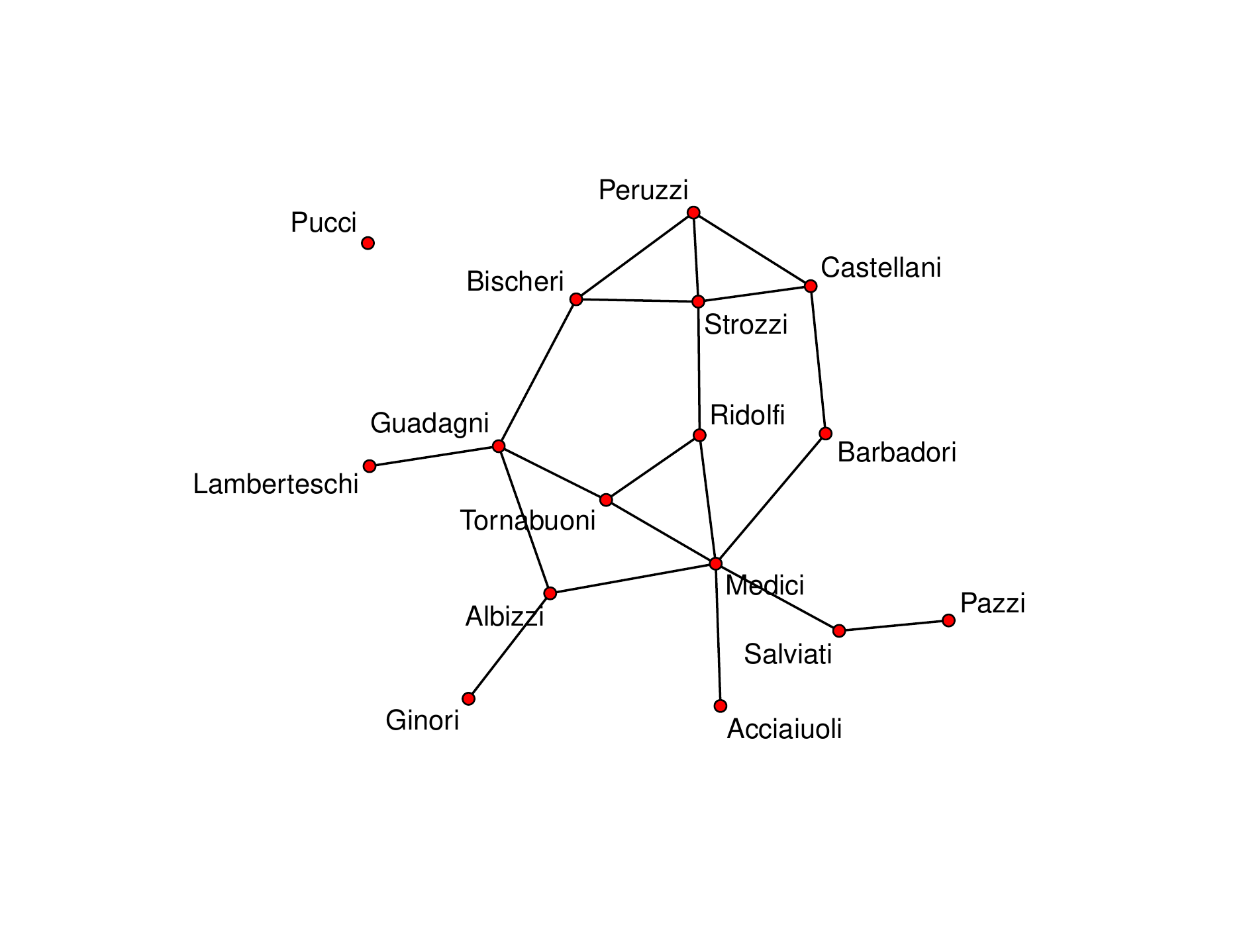}
\vspace{-2cm}
\caption{Marriage relations between Florentine families} \label{florentine}
\end{figure} 

Using the R package \emph{sna}\footnote{\url{https://cran.r-project.org/web/packages/sna/sna.pdf}},
we fit an exponential random graph model to these data with $k=2$ and $H_2$ a $2$-star.  
The fitted coefficients are $\beta_1= -1.6339$ and $\beta_2=0.0049*2=0.0098$, and so Proposition~\ref{prop:2star} applies. 
The function $\Phi$ is
\be{
\Phi(a) = -1.6339 + 0.0196a,
}
and the unique solution to $\phi(a)=a$ is $a^* = 0.036743$.
Thus the upper bound of Proposition~\ref{prop:2star} between the fitted model and an \ER\ graph with parameter $a^*$ is approximately
\be{
\binom{16}{2}\frac{0.036743\sqrt{8(0.0098)(1-0.036743)}}{\bclr{4(1-(0.0098)}\sqrt{14}}={0.0817595}.
}

\medskip

Our next result is a generalisation of Theorem~\ref{thm:smallbetas}, which applies to the whole high temperature regime, meaning  there is a unique minimiser to~\eq{eq:posbvar}.
The result is more general, but the constants are not explicit.
\begin{theorem}\label{thm:hightemp}
For given $\beta_1\in\IR$, $\beta_\ell>0$, $\ell=2,\ldots,k$, assume there is a unique $\astar\in[0,1]$ that satisfies both $\astar=\phi(\astar)$ and $\phi'(\astar)<1$.
Then for $X\sim\ergm(\beta)$, $Z$ having the~\ER\ distribution with parameter~$\astar$, 
and any function $h:\{0,1\}^{\binom{n}{2}}\to \IR$, there is a constant $C$ depending only on $\beta$ and $H_1,\ldots, H_\ell$, such that
\be{
\abs{\IE h(X)-\IE h(Z)}
        \leq   C \norm{\Delta h} n^{3/2}.
}
\end{theorem}

Now, to compare our ERGM results to those existing, we
already mentioned the result of \cite{Chatterjee2013}, which shows 
large ERGMs are close to finite mixtures of \ER\ random graphs. Theorems~\ref{thm:smallbetas} and~\ref{thm:negbetas} explicitly quantify this in a subset of the high temperature regime,
while Theorem~\ref{thm:hightemp} gives rates over the whole high temperature regime.
Our results rely on a close analysis
of the ERGM Glauber dynamics, following \cite{Bhamidi2011}, where mixing times are studied. In particular, they find that in the high temperature regime, the chain mixes at the same order ($n^2 \log(n)$) as the \ER\ case. 
The region described in our Theorem~\ref{thm:smallbetas} is a subset of the high temperature regime
where the Glauber dynamics are contracting.  
For parameter values where there is more than one \emph{local} maximum,  \cite{Bhamidi2011} show that
the chain takes exponential (in~$n$) time to mix. One approach to adapt our results to this case is to study the Glauber dynamics of the ERGM conditioned to be in a region ``close" to a local maximum where the dynamics are contracting; a suitable region can be read from Lemma~\ref{lem:gdsetbd}.
We also mention \cite{Radin2013}, where the variational problem~\eq{eq:posbvar} is analysed in the case when $k=2$ and $\beta_2>0$ or where $\beta_2<0$ in the case where $H_2$ is a star; even this basic case is important to understand, though not straightforward to analyse. Further analysis of the variational problem for $k=3$ is given in \cite{Yin2013}.

Closest to our $\ergm$ results is 
\cite{Eldan2017} (already mentioned in an example above), 
and in particular they showed in their Theorem~18 
that for positive $\beta_\ell$ for $\ell=2,\ldots,k$ and
assuming there is a unique solution~$\astar$ to~\eq{eq:posbvarcrit} and with $\phi'(\astar)<1$ (the high temperature regime) then 
for $X\sim\ergm(\beta)$, $Z$ having the~\ER\ distribution with parameter~$\astar$, 
and any $h:\{0,1\}^{\binom{n}{2}}\to \IR$,
\be{
\abs{\IE h(X)-\IE h(Z)}
	\leq \norm{\Delta h} \bigo(n^{2(1-\theta)})
}
for some $0<\theta<1/16$, and the constant can be made numerically explicit in fixed examples. The result is more general than Theorem~\ref{thm:smallbetas}, as it covers the entire high temperature regime, but the rate is not as good. Theorem~\ref{thm:hightemp} covers the high temperature regime with a better rate than \cite[Theorem~18]{Eldan2017}, but without explicit constants (and really no hope of extracting them from our proof). 
%

The remainder of this paper is organized as follows. We finish the introduction with a summary discussion. In Section~\ref{sec:stn} we state and prove our main
general approximation result, Theorem~\ref{thm:key}, and show how it implies Theorem~\ref{thm:key1}. In Section~\ref{sec:ergm}, we state and prove our ERGM results, Theorems~\ref{thm:smallbetas} and~\ref{thm:negbetas}, and Propositions~\ref{prop:2star} and~\ref{prop:tri}, and Theorem~\ref{thm:hightemp}.

\subsection{Discussion}

Our main result gives an easy to compute upper bound between a vector of Bernoulli variables that are the stationary distribution of a fast mixing Markov chain and an independent vector. In typical applications, such as Ising models and ERGMs,  the result gives approximation with rates of convergence for expectations of certain statistics of the system. Our work further supports the notion, appearing in other contexts, that stationary distributions of fast mixing Markov chains should have properties similar to independent systems. Our bounds are simple and explicit, and improve upon others given in \cite{Eldan2017a, Eldan2017}, which in turn are derived from a ``non-linear" large deviations framework developed in \cite{Eldan2016}. That framework uses a notion of complexity referred to as \emph{Gaussian width}. A similar framework with a different notion of complexity is given in \cite{Chatterjee2017}; see also \cite{Chatterjee2016}. It would be interesting to make the connection between these complexity measures and fast mixing explicit, which would complete more of the picture around concentration, mixing, and distance to independence. 

Though our results only apply in high temperature regimes, note that  in low temperature regimes, fast mixing may occur for the measure conditioned to be in a ``good" neighbourhood. For example, in the $\ergm$, we would condition on being in the set given by~\eq{eq:subctbd} for small enough $\eps$ such that~\eq{eq:contractbd} is satisfied. The difficulty is to ensure that the conditional Glauber dynamics are well-behaved at the boundary.

\begin{remark}
After the second author spoke about this work at the MIT probability seminar on 23 October 2017, we became aware that Guy Bresler and Dheeraj  Nagaraj have been  working independently on a similar general approximation result, with applications to Ising models; see \cite{Bresler2017}.
In particular, they prove Lemmas~\ref{lem:genapp} and~\ref{lem:genapp1}  using roughly the same ideas as ours, and explore their use in Ising models, going beyond the illustrative Section~\ref{sec:ising}. 
Results for ERGMs, which are our main application,  are not developed in their work. For added transparency we have co-ordinated the submissions of our two papers. 
\end{remark}

\section{General approximation result}\label{sec:stn}

In this section we state and prove our general approximation result,
and then show how it implies Theorem~\ref{thm:key1}.
As previously mentioned, we work in the setup of \cite{Hayes2006} and  \cite{Dyer2009}. Define the  $N\times N$ \emph{influence matrix} $\hat R$ for the Glauber dynamics of $\law(X)$ by
\be{
\hat R_{r s}:=\max_{x\in\{0,1\}^N}\bbabs{q_X\bclr{\clr{x\s{s,1}}\s{r,1}|x\s{s,1}}-q_X\bclr{\clr{x\s{s,0}}\s{r,1}|x\s{s,0}}}.
}
Then $\hat R_{r s}$ is the maximum amount that the conditional distribution of the $r^{th}$ coordinate of $x$ can change due to a change in the $s^{th}$ coordinate of $x$. For $1\leq p\leq \infty$, let $\norm{\cdot}_p$ be the $p$-norm on $\IR^N$, and define the matrix operator $p$-norm
\be{
\norm{A}_p:=\sup_{v\not=0} \frac{\norm{A v}_p}{\norm{v}_p};
}
as is typical, the notation $\norm{\cdot}_p$ has different meaning depending on whether the argument is a matrix or vector, but this should not cause confusion.

\begin{theorem}\label{thm:key}
Let $X,Y\in\{0,1\}^N$ be random vectors, $h:\{0,1\}^N\to\IR$, and assume the continuous time Glauber dynamics for~$\law(X)$ is irreducible.   For $s\in[N]$, set $c_s:=\norm{\Delta_s h}$ and 
$v_s(y):=\abs{q_X(y\s{s,1}|y)- q_Y(y\s{s,1}|y)}$, and $c:=(c_1,\ldots, c_N)$ and $v(Y):=(v_1(Y),\ldots, v_N(Y))$.  
Assume there is an $N\times N$ 
matrix $R$, such that for all $r,s\in[N]$, and some $1\leq p \leq \infty$ and $\eps=\eps_p>0$,
\ben{\label{dep}
\hat R_{rs}\leq R_{rs},  \, \mbox{ and } \, \,  \norm{R}_p\leq 1-\eps<1;
}
where~$\hat R$ is the influence matrix for the Glauber dynamics of~$\law(X)$.
Then for $q:=p/(p-1)$,
\be{
\babs{\IE h(X)-\IE h(Y)}  
\leq  \eps^{-1} \norm{c}_{q}  \, \IE \norm{v(Y)}_p.
}
\end{theorem}

\begin{remark}
A matrix $R$ such that \eqref{dep} holds is also called a {\it dependency matrix}.
The condition $\norm{R}_p\leq 1-\eps<1$ with $p=\infty$ (maximum of row sums) is the \emph{Dobrushin} condition from \cite{Dobrushin1970}, and with $p=1$ (maximum of column sums) is the \emph{Dobrushin-Shlosman} condition  from \cite{Dobrushin1985}.
Variations on the bound may be obtained for more general matrix norms, following the approach of \cite{Dyer2009}, but we pursue these generalisations and their implications elsewhere.
\end{remark}

The proof of the theorem uses Stein's method, due to \cite{Stein1972}, \cite{Stein1986};  see \cite{BarbourChen2005}, \cite{Chen2011}, \cite{Ley2017}, \cite{Ross2011} for various introductions; and we follow Barbour's generator method from \cite{Barbour1988, Barbour1990}. Our approach is
related to ideas from \cite{Reinert2005} for functions of independent variables, and \cite{Eichelsbacher2008} for one-dimensional Gibbs measures. 
The generator approach requires bounds on certain mixing quantities, and these are obtained from the bound on the norm of the influence matrix, following \cite{Hayes2006},
\cite{Dyer2009}, and \cite{Bubley1997}. 
We also remark that the key
 Lemma~\ref{lem:delf} below is closely related to ideas in \cite[Section~4.2]{Chatterjee2005a}, which gives fast mixing criteria for concentration of Lipschitz functions of stationary distributions.

The underlying idea to the generator approach is that a distribution $\mu$ is characterised as the stationary distribution of a Markov process which has generator $\cA_\mu$. The generator $\cA$ appears in the so-called {\it Stein equation} 
\be{
\cA_\mu f(x) = h(x) -  {\mathbbm{E}} h(X)
}
where $X$ has distribution $\mu$ and $h$ is a test function. Two distributions can then be compared through the comparison of the corresponding generators;  if $Y$ has distribution $\nu$ then 
\be{
{\mathbbm{E}} h(Y ) -  {\mathbbm{E}} h(X) = {\mathbbm{E}} \cA_\mu f(Y) 
=  {\mathbbm{E}} \cA_\mu f(Y)  -  {\mathbbm{E}} \cA_\nu f(Y) .
}
Here the last equality follows from ${\mathbbm{E}} \cA_\nu f(Y)=0$ under some regularity assumptions.

{A Markov process which has stationary distribution $\mu$} can sometimes be obtained through a sequence of exchangeable operations on the configurations which the chain can take on; see for example  \cite{Rinott1997} and \cite{Reinert2005}.
In this paper the exchangeable construction is provided by the Glauber dynamics with an exponential clock for the transition times. 

The first lemma we need in order to prove the theorem is by now standard in Stein's method; it provides a solution of the Stein equation induced by the generator of the Glauber dynamics.
Define the continuous time Glauber dynamics with exponential rate~1 holding times,
having generator~$\cA:=\cA_X$ given by
\ben{\label{eq:gibbgen}
\cA f(x)=\frac{1}{N}\sum_{s\in[N]}\bbcls{q(x\s{s,1}|x)\Delta_{s} f(x)+  \bclr{f(x\s{s,0})-f(x)}}.
}

\begin{lemma}\label{lem:genapp}
Let $X\in\{0,1\}^N$ be a random vector and $(X(t))_{t\geq0}$ be the continuous time Glauber dynamics Markov chain for~$\law(X)$
with generator~\eq{eq:gibbgen}. 
If the Markov chain is irreducible, then for any $h:\{0,1\}^N\to \IR$, the function
 \be{
f_h(x):=-\int_{0}^\infty \IE\cls{h(X(t))-\IE h(X)|X(0)=x} dt
} 
is well-defined and satisfies $\cA f_h(x)= h(x)-\IE h(X)$. 
\end{lemma}


The next lemma gives an intermediate bound on the quantity we want to approximate, through comparing the generators of the Glauber dynamics \eqref{eq:gibbgen} for the two processes under investigation.
\begin{lemma}\label{lem:genapp1}
Let $X,Y\in\{0,1\}^N$ be random vectors, $h:\{0,1\}^N\to\IR$, $f_h$ be as in Lemma~\ref{lem:genapp}. 
Then
\bes{
\babs{\IE h(X)-\IE h(Y)} &\leq  \frac{1}{N}\sum_{s\in[N]}\IE \bbcls{\babs{q_X(Y\s{s,1}|Y)- q_Y(Y\s{s,1}|Y)}\babs{\Delta_{s} f_h(Y)}}.
}
\end{lemma}
\begin{proof}
From Lemma~\ref{lem:genapp} and the fact that $\IE \cA_Y f(Y)=0$,
we have 
\ba{
\babs{\IE h(X)-\IE h(Y)}&=\babs{\IE \cA_X f_h(Y)} \\
	&=\babs{\IE \cA_X f_h(Y)-\IE \cA_Y f_h(Y)}\\
	&= \frac{1}{N}\bbbabs{\sum_{s\in[N]}\IE \bbcls{\bclr{q_X(Y\s{s,1}|Y)- q_Y(Y\s{s,1}|Y)}\Delta_{s} f_h(Y)}} \\
	&\leq \frac{1}{N}\sum_{s\in[N]}\IE \bbcls{\babs{q_X(Y\s{s,1}|Y)- q_Y(Y\s{s,1}|Y)}\babs{\Delta_{s} f_h(Y)}}. \qedhere
}
\end{proof}

The next lemma bounds $\abs{\Delta_{s} f_h(Y)}$.
\begin{lemma}\label{lem:delf}
Let $X,Y\in\{0,1\}^N$ be random vectors, $h:\{0,1\}^N\to\IR$, $f_h$ be as in Lemma~\ref{lem:genapp}.
For each $s\in[N]$, $x\in\{0,1\}^N$, and $m\geq0$, let~$(U\t{x,s}(m), V\t{x,s}(m))$ be any coupling such that
${\mathcal{L}} (U\t{x,s}(m)) = {\mathcal{L}}(X(m) | X(0)=x\s{s,1})$ and ${\mathcal{L}} (V\t{x,s}(m)) =  {\mathcal{L}}(X(m) | X(0)=x\s{s,0})$.
Then
\bes{
\abs{\Delta_s f_h(x)} &\leq 
\mathop{\sum_{r\in[N]}}_{m\geq0} \norm{\Delta_r h} \IP\bclr{U\t{x,s}_r(m)\not=V\t{x,s}_r(m)}.
}
\end{lemma}
\begin{proof}
 Let $0=T_0<T_1<T_2<\cdots$ be the jump times of the continuous time Glauber dynamics Markov chain.
Since the jump times are independent of the discrete skeleton, and using Lemma~\ref{lem:genapp}, we have for $x\in\{0,1\}^N$,
\ba{
\abs{\Delta_s f_h(x)}&\leq \int_{0}^\infty \bbabs{\IE \cls{h(X(t))|X(0)=x\s{s,1}}-\IE \cls{h(X(t))|X(0)=x\s{s,0}}} dt \\
		&\leq \sum_{m=0}^\infty \IE \bbbabs{\int_{T_m}^{T_{m+1}}  \bcls{h(U\t{x,s}(m))-h(V\t{x,s}(m))}dt}\\
		&=\sum_{m=0}^\infty \IE  \babs{h(U\t{x,s}(m))-h(V\t{x,s}(m))} \\
		&\leq \mathop{\sum_{r\in[N]}}_{m\geq0}\norm{\Delta_r h} \, \IE  \babs{U\t{x,s}_r(m)-V\t{x,s}_r(m)}\\
	&= \mathop{\sum_{r\in[N]}}_{m\geq0}\norm{\Delta_r h} \, \IP\bclr{U\t{x,s}_r(m)\not=V\t{x,s}_r(m)}. \qedhere
}
\end{proof}


{The approach for the proof of Theorem~\ref{thm:key} is closely related
to that of \cite[Proof of Theorem~4.3]{Chatterjee2005a}.}
\begin{proof}[Proof of Theorem~\ref{thm:key}]
Set 
\be{
B:=\left(1-\frac{1}{N}\right) I +  \frac{R}{N},
}
where $I$ denotes the identity matrix. {We follow the path coupling paradigm of \cite{Bubley1997}. A}
direct consequence of \cite[Formula~(4)]{Dyer2009}--note that their $R$ is the transpose of ours--
implies {that} 
there is a coupling of the Glauber dynamics of~$\law(X)$ such that
\be{
\IP\bclr{U\t{x,s}_r(m)\not=V\t{x,s}_r(m)}\leq \bclr{B^m}_{rs}.
}
Therefore by Lemma~\ref{lem:delf}, 
\bes{
\abs{\Delta_s f_h(x)} &\leq 
\mathop{\sum_{r\in[N]}}_{m\geq0} \norm{\Delta_r h} \bclr{B^m}_{rs}.
}
Applying this in Lemma~\ref{lem:genapp1}, and then using H\"older's inequality and sub-multiplicativity of matrix norms, we find
\bes{
\babs{\IE h(X)-\IE h(Y)}&\leq  N^{-1}  \mathop{\sum_{r,s\in[N]}}_{m\geq0} \IE \left[c_r (B^m)_{rs} v_s(Y)\right] \\
		&\leq N^{-1}\sum_{m\geq0} \norm{c}_q \, \IE \norm{B^m v(Y)}_p\\
		&\leq N^{-1} \norm{c}_q \, \IE \norm{v(Y)}_p \sum_{m\geq0}   \norm{B}^m_p. 
}
The result now follows after noting
\be{
\norm{B}_p=\bbbnorm{\bbclr{1-\frac{1}{N}} I +  \frac{1}{N}R\,}_p
{\leq \frac{N-1}{N} + \frac{1}{N} \norm{R}_p} \leq 1-\frac{\eps}{N}. \qedhere
}
\end{proof}

Specialising the proof of Theorem~\ref{thm:key} to the case where $p=1$,~$Y$ is a 
vector of independent Bernoulli variables, and writing~$X$ in Gibbs measure form, yields Theorem~\ref{thm:key1}, as we now show.
\begin{proof}[Proof of Theorem~\ref{thm:key1}]
Since
\bes{
 \IE \, \dham(\tilde U\t{x,s}, \tilde V\t{x,s})\leq (1-\rho),
}
path coupling implies that for each $s\in[N]$, $x\in\{0,1\}^N$, and $m\geq0$, 
there is a coupling~$(U\t{x,s}(m), V\t{x,s}(m))$  such that
${\mathcal{L}} (U\t{x,s}(m)) = {\mathcal{L}}(X(m) | X(0)=x\s{s,1})$, and ${\mathcal{L}} (V\t{x,s}(m)) =  {\mathcal{L}}(X(m) | X(0)=x\s{s,0})$,
and
\be{
\IE \, \dham\bclr{U\t{x,s}(m), V\t{x,s}(m)}\leq (1-\rho)^m.
}
Thus, from Lemma~\ref{lem:delf},
\bes{
\abs{\Delta_s f_h(x)} &\leq 
\mathop{\sum_{r\in[N]}}_{m\geq0} \norm{\Delta_r h} \IP\bclr{U\t{x,s}_r(m)\not=V\t{x,s}_r(m)} \\
	&\leq \norm{\Delta h}\sum_{m\geq 0} \IE \, \dham\bclr{U\t{x,s}(m), V\t{x,s}(m)}\\
	&= {\norm{\Delta h}} \rho^{-1}.
}
Now, Lemma~\ref{lem:genapp1} implies
\bes{
\babs{\IE h(X)-\IE h(Y)} &\leq  \frac{\norm{\Delta h}}{N\rho}\IE \norm{v(Y)}_1,
}
where we use the notation from Theorem~\ref{thm:key}.
To bound 
$\IE \norm{v(Y)}_1$, first we write the Glauber dynamics probabilities for $\law(X)$ as follows.
\begin{eqnarray*} 
q_X\bclr{x\s{s,1}|x}&=& \IP\bclr{X_s=1|(X_u)_{u\not=s}=(x_u)_{u\not=s}}\\
&=&\frac{ \IP\bclr{X_s=1;  (X_u)_{u\not=s}=(x_u)_{u\not=s}} }{\IP\bclr{X_s=1;  (X_u)_{u\not=s}=(x_u)_{u\not=s}} + \IP\bclr{X_s=0;  (X_u)_{u\not=s}=(x_u)_{u\not=s}}} \\
&=& \frac{ \exp\{ L(x^{(s,1)} ) \} }{ \exp \{ L(x^{(s,1)}) \}  +  \exp\{  L(x^{(s,0)}) \} }, \\
\end{eqnarray*}
which simplifies to 
\be{
q_X(x\s{s,1}|x)=\frac{e^{\Delta_s L(x)}}{e^{\Delta_s L(x)}+1}
=\frac{1+\tanh\bclr{\frac{1}{2} \Delta_s L(x)}}{2}.
}
Now, it is easy to see that
\ben{\label{eq:tanlog}
p_s=\frac{1+\tanh\bclr{\frac{1}{2} \log\bclr{\frac{p_s}{1-p_s}}}}{2},
}
and $\tanh$ is $1$-Lipschitz, so, noting that $q_Y(x\s{s,1}|x)=p_s$,
\be{
\bbabs{q_X(Y\s{s,1}|Y)-p_s}\leq \frac{1}{4}\bbbabs{\Delta_s L(Y)-\log\bbbclr{\frac{p_s}{1-p_s}}}.
}
A simple rearrangement using~\eq{eq:tanlog} shows that any solution~$p\in[0,1]^N$ to the system~\eq{eq:fp2} also satisfies
\be{
\log\bbbclr{\frac{p_s}{1-p_s}}=\IE \Delta_s L(Y),
}
and the result now easily follows.
\end{proof}

\section{Proof of ERGM results}\label{sec:ergm}

To prove Theorem~\ref{thm:smallbetas}, we use Theorem~\ref{thm:key}.
In the notation there, we want to bound the influence matrix. 
The ideas we use essentially come from \cite{Bhamidi2011},
where mixing times for the Glauber dynamics of ERGMs are derived. Their results are not strong enough to apply directly, but we are able to derive sharper, more quantitative versions.

For a graph $H$ with~$v_H$ vertices and~$e_H$ edges, define
\be{
r_H(x;ij)=\left(\frac{\Delta_{ij} t(H,x)}{2  e_H n(n-1)\cdots (n-v_H+3)}\right)^{1/(e_H-1)},
}
and $r_\ell(x;ij):=r_{H_\ell}(x;ij)$. If $Z$ is an \ER\ graph with parameter~$\astar$,
then for any small subgraph $H$ with at least two edges, with high probability, $r_H(Z;ij)=\astar+\lito(1)$. The basic idea of the proof is that if 
$r_H(x;ij)\approx \astar$, then the $\ergm(\beta)$ Glauber dynamics starting from~$x$ are similar to the \ER\ dynamics.  

The next lemma makes explicit some statements of  \cite[Proof of Lemma~18]{Bhamidi2011}. 
For a graph $H$ and $e\in E(H)$, define the graph $H\setminus e$ 
to be $H$ with the edge $e$ removed, but with all vertices retained, even if isolated.
{For a polynomial $f(x) = \sum_{i=1}^k a_k x^k$ we use the notation
$|f|(x) = \sum_{i=1}^k |a_k | x^k$.} 
\begin{lemma}\label{lem:gdsetbd}
For all $x\in\{0,1\}^{\binom{n}{2}}$, 
\ben{\label{eq:negsubctbd}
\sum_{st\not=ij} \bbabs{q\bclr{x\s{st,1}|x\s{ij,1}}-q\bclr{x\s{st,1}|x\s{ij,0}}}\leq \tsfrac{1}{2}\abs{\Phi}'(1).
}
Assuming now $\beta_\ell>0$ for $\ell=2,\ldots,k$, if $\astar\in[0,1]$ satisfies $\astar=\phi(\astar)$ and $x\in\{0,1\}^{\binom{n}{2}}$
is such that for all edges $st\not=ij$ and all graphs~$H=H_\ell\setminus e$ for some $\ell=2,\ldots,k$ and $e\in E(H_\ell)$, 
\ben{\label{eq:subctbd}
 \max_{st\not= ij} \bbclc{\babs{r_H(x\s{ij,1}; st)-\astar}\vee\babs{r_H(x\s{ij,0}; st)-\astar}}\leq \eps,
}
then
\ban{
\sum_{st\not=ij} &\babs{q\bclr{x\s{st,1}|x\s{ij,1}}-q\bclr{x\s{st,1}|x\s{ij,0}}} \notag \\
	&\leq \frac{1}{2}\left[\Phi'(\astar) + \eps \Phi''(1)\right]\bbclr{1 \wedge \left[\mathrm{sech}^2\bclr{\Phi(\astar)} +  C_2 \left(\eps +n^{-1}\right) \Phi'(1)\right]}.
	\label{eq:contractbd}
}
\end{lemma}
\begin{proof}
{For $st\not=ij$}, we bound
\ba{
&\bbabs{ q\bclr{x\s{st,1}|x\s{ij,1}}-q\bclr{x\s{st,1}|x\s{ij,0}}}\\
&=\frac{1}{2}\left|\tanh\bbbclr{\frac{1}{2}\sum_{\ell=1}^k \beta_\ell \bclr{\Delta_{st} t_\ell(x\s{ij,1})}}-\tanh\bbbclr{\frac{1}{2}\sum_{\ell=1}^k \beta_\ell \bclr{\Delta_{st} t_\ell(x\s{ij,0})}}\right| \\
&\leq \frac{1}{4}
\bbbclr{\sum_{\ell=1}^k\abs{\beta_\ell} \bclr{\Delta_{st} \Delta_{ij} t_\ell(x)}} \\
&\qquad\qquad\times
\bbbclr{1\wedge 
\bbbcls{\mathrm{sech}^2
\bbbclr{\frac{1}{2}\sum_{\ell=1}^k \beta_\ell \bclr{\Delta_{st} t_\ell(x\s{ij,0})}}+ \frac{C_{2}}{4}\bbbclr{\sum_{\ell=1}^k \abs{\beta_\ell} \Delta_{st} \Delta_{ij} t_\ell(x)}}},
}
where the inequality follows by using first and second order Taylor's theorem for $\tanh$.
For the first assertion,  {first} we use the~$1$ in the minimum and sum over edges $st\not=ij$. {Next}
 recall that   $\Delta_{st}\Delta_{ij}t(H_\ell,x)$
is the number of $(x\s{ij,1})\s{st,1}$ edge-preserving injections of~$V(H_\ell)$ into $V(x)$, which use \emph{both} edges~$ij$ and~$st$. Writing $t(H_\ell, x; e\mapsto st)$ for the number of these injections that map~$e$ to~$st$ in~$x$, we have
\be{
\Delta_{st} \Delta_{ij} t(H_\ell, x)=\sum_{e\in H_\ell}
{\frac{ \Delta_{ij} t(H_\ell, x; e\mapsto st)}{n(n-1)\cdots(n-v_\ell+3)}} 
}
and then 
\besn{\label{eq:4k4}
\sum_{st\not=ij} \bclr{\Delta_{st} \Delta_{ij} t_\ell(x)}
	&=\sum_{st\not=ij} \sum_{e\in E(H_\ell)} \frac{ \Delta_{ij} t(H_\ell, x; e\mapsto st)}{n(n-1)\cdots(n-v_\ell+3)}\\
	&=\sum_{e\in H_\ell} \frac{\Delta_{ij}t(H_\ell\setminus e, x)}{n(n-1)\cdots(n-v_\ell+3)}.
}
where the first equality is by swapping summations; c.f.~\cite[Lemma~10]{Bhamidi2011}.
Now, the first assertion follows since
\be{
\frac{\Delta_{ij}t(H_\ell\setminus e, x)}{n(n-1)\cdots(n-v_\ell+3)}\leq\frac{\Delta_{ij}t(H_\ell\setminus e, \bf{1})}{n(n-1)\cdots(n-v_\ell+3)}\leq 2(e_\ell-1),
}
where $\textbf{1}$ denotes the complete graph. {For the last inequality, note that $\Delta_{ij}t(H_\ell\setminus e, \bf{1})$ is the number of injections from $V(H_\ell \setminus e$) into $V(\bf{1})$ using edge $ij$. There are at most $ 2(e_\ell-1)$ of assigning vertices of $H_\ell \setminus e$ to $ij$, and there are at most $n(n-1)\cdots(n-v_\ell+3)$ ways to assign the remaining vertices. Hence the inequality follows.}
 
{For the second assertion,} assume that $\beta_\ell>0$ for $\ell=2,\ldots,k$. {For an alternative bound on $\sum_{\ell=1}^k {\beta_\ell} \bclr{\Delta_{st} \Delta_{ij} t_\ell(x)}$}, note that
\be{
\Delta_{st} t_\ell(x\s{ij,0})=2 e_\ell r_\ell(x\s{ij,0}; st)^{e_\ell-1}.
}
{Hence} by~\eq{eq:subctbd} and the fact that; using the positivity of $\beta_2,\ldots,\beta_k$; $0\leq \Phi'(a) \le \Phi'(1) = \sum_{\ell=2}^k \beta_\ell e_\ell(e_\ell-1)$
for $0<a<1$, we have
\ban{\nonumber
\bbbabs{\frac{1}{2}\sum_{\ell=1}^k \beta_\ell \bclr{\Delta_{st} t_\ell(x\s{ij,0})}-\Phi(\astar)}
&={\bbbabs{\sum_{\ell=1}^k \beta_\ell e_\ell \bbclr{r_\ell(x\s{ij,0}; st)^{e_\ell-1} -(\astar)^{e_\ell-1}}}} \\ \label{firstineq}
&\leq \eps \sum_{\ell=2}^k \beta_\ell e_\ell(e_\ell-1);
}
we {shall derive \eqref{firstineq}, and use it} a few more times, below.
Then we can write
\be{
\mathrm{sech}^2
\bbbclr{\frac{1}{2}\sum_{\ell=1}^k \beta_\ell \bclr{\Delta_{st} t_\ell(x\s{ij,0})}}
\leq \mathrm{sech}^2\bclr{\Phi(\astar)} +  \eps  C_2 
\sum_{\ell=2}^k \beta_\ell e_\ell(e_\ell-1).
}
We also bound
\ben{\label{secondineq}
\Delta_{st} \Delta_{ij} t_\ell(x)\leq \Delta_{st} \Delta_{ij} t_\ell(\textbf{1})
\leq \frac{4 e_\ell (e_\ell-1)}{n}.
}
 To see {\eqref{firstineq} and \eqref{secondineq}}, first note that for $\ell=2,\ldots,k$, following the same argument as for Remark \ref{rem:var}, 
 $\Delta_{st}\Delta_{ij}t(H_\ell,x)$
is the number of $(x\s{ij,1})\s{st,1}$ edge-preserving injections of~$V(H_\ell)$ into $V(x)$, which use \emph{both} edges~$ij$ and~$st$. So Inequality~\eqref{firstineq}follows from the fact that such subgraph injections 
increase with the number of  edges, maximized for the complete graph $\textbf{1}$.
For Inequality~\eqref{secondineq}, $\Delta_{st}\Delta_{ij}t(H_\ell,\textbf{1})$ is simply the number of
injections from $V(H_\ell)$ into $V(\textbf{1})$ using both edges $st$ and $ij$.  
To count the number of injections, there are at most $4e_\ell (e_\ell-1)$ ways to assign vertices of $H_\ell$ to $st$ and $ij$ (and maybe fewer depending on the topology of $H_\ell$ and whether $st$ and $ij$ share a vertex), and then there are at most $(n-3)\cdots(n-v_\ell+1)$ ways to assign the remaining vertices. The remaining vertices can be counted: $v_\ell-3$ vertices of $H_\ell$ to $(n-3)$ remaining vertices of the big graph, when $st$ and $ij$ share a vertex, or $v_\ell-4$ vertices of $H_\ell$ to $(n-4)$ remaining vertices otherwise. Thus we find
\be{
\Delta_{st} \Delta_{ij} t_\ell(\textbf{1})\leq 4e_\ell(e_\ell-1) \frac{(n-3)\cdots(n-v_\ell+1)}{n\cdots (n-v_\ell+3)}\leq \frac{4e_\ell(e_\ell-1)}{n},
}
as desired.

Combining the bounds of the previous displays, we have
\ba{
\bbabs{q&\bclr{x\s{st,1}|x\s{ij,1}}-q\bclr{x\s{st,1}|x\s{ij,0}}}\\
&\leq
\bbbclr{\sum_{\ell=1}^k \beta_\ell \bclr{\Delta_{st} \Delta_{ij} t_\ell(x)}}
\bbbclr{1\wedge 
\bbbcls{\frac{\mathrm{sech}^2\bclr{\Phi(\astar)}}{4} +  \bbbclr{\eps +\frac{1}{n}}\frac{ C_2}{4} 
\sum_{\ell=2}^k \beta_\ell e_\ell(e_\ell-1)
}
}.
}
Now summing over edges $st\not=ij$, using~\eq{eq:4k4},~\eq{eq:subctbd}, and the mean value theorem, we find
\be{
\sum_{st\not=ij} \bclr{\Delta_{st} \Delta_{ij} t_\ell(x)}\leq \sum_{e\in H_\ell} \frac{\Delta_{ij}t(H_\ell\setminus e, x)}{n(n-1)\cdots(n-v_\ell+3)}\leq 2e_\ell(e_\ell-1)\left((\astar)^{e_\ell-2} + \eps(e_\ell-2)\right).
}
We now have
\be{
\sum_{\ell=1}^k \beta_\ell\sum_{st\not=ij} \bclr{\Delta_{st} \Delta_{ij} t_\ell(x)}
	\leq 2 \Phi'(\astar) +2 \eps \Phi''(1),
}
and combining the bounds above yields the lemma.
\end{proof}

We prove a quantitative version of \cite[Lemma~18]{Bhamidi2011}.
\begin{lemma}\label{lem:hambd1}
If $x\in\{0,1\}^{\binom{n}{2}}$ satisfies the hypotheses of Lemma~\ref{lem:gdsetbd}, then there is a coupling of two realisations  of one step of the embedded discrete time Markov chain for the Glauber dynamics for $\ergm(\beta)$, denoted
$\bclr{U\t{x,ij}(m),V\t{x,ij}(m)}_{m=0,1}$, such that $U\t{x,ij}(0)=x\s{ij,1}$, $V\t{x,ij}(0)=x\s{ij,0}$, and
\ba{
\IE \dham\bclr{U\t{x,ij}(1),V\t{x,ij}(1)}\leq (1-\rho),
}
where
\be{
\rho=\frac{1-\tsfrac{1}{2}\abs{\Phi}'(1)}{\binom{n}{2}}.
}
If $\beta_\ell>0$ for $\ell=2,\ldots,k$, then the same inequality holds with 
\be{
\rho=\frac{1-\frac{1}{2}\left[\Phi'(\astar) + \eps \Phi''(1)\right]\bclr{1 \wedge \left[\mathrm{sech}^2\bclr{\Phi(\astar)} +  C_2 \left(\eps +n^{-1}\right) \Phi'(1)\right]}}{\binom{n}{2}}.
}
\end{lemma}
\begin{proof}
At the first step, we choose a common edge at random and update with the optimal coupling of Bernoulli variables, which is unsuccessful for edge~$st$ with probability
$\babs{q\bclr{x\s{st,1}|x\s{ij,1}}-q\bclr{x\s{st,1}|x\s{ij,0}}}$.
Now note that
\be{
\dham\bclr{U\t{x,ij}(1),V\t{x,ij}(1)}-\dham\bclr{U\t{x,ij}(0),V\t{x,ij}(0)}
} is~$-1$ if edge~$ij$ is chosen, is $1$ if another edge is chosen and the coupling fails, and is otherwise~$0$.
Thus
\be{
\IE \dham\bclr{U\t{x,ij}(1),V\t{x,ij}(1)}=1-\binom{n}{2}^{-1} + \binom{n}{2}^{-1}\sum_{st\not=ij} \bbabs{q\bclr{x\s{st,1}|x\s{ij,1}}-q\bclr{x\s{st,1}|x\s{ij,0}}},
}
and the result now follows immediately from Lemma~\ref{lem:gdsetbd}.
\end{proof}

With these lemmas, we can prove Theorem~\ref{thm:smallbetas}.
\begin{proof}[Proof of Theorem~\ref{thm:smallbetas}]
Note that $0\leq r_H(y;st)\leq 1$ for any graph $H$ and $y\in\{0,1\}^{\binom{n}{2}}$, and so setting $\eps=A^*$, we have that~\eq{eq:subctbd} is automatically satisfied
and Lemma~\ref{lem:hambd1} applies with~$\rho=\gamma/\binom{n}{2}>0$; the positivity by the assumption of the theorem. Applying Theorem~\ref{thm:key1} and using symmetry of the edges proves the result.
\end{proof}

\begin{proof}[Proof of Theorem~\ref{thm:negbetas}]
Lemma~\ref{lem:hambd1} applies with $\rho=\bclr{1-\tsfrac{1}{2}\abs{\Phi}'(1)}/\binom{n}{2}>0$, so Theorem~\ref{thm:key1} and edge symmetry yields the result.
\end{proof}

\begin{proof}[Proof of Proposition~\ref{prop:2star}]
Let $k=2$ and $H_2$ be a two-star. Applying Theorem~\ref{thm:negbetas}, we compute
\be{
\abs{\Phi}(a)= \abs{\beta_1} +2 \abs{ \beta_2 } a, \hspace{5mm}
\abs{\Phi}'(a)=2\abs{ \beta_2}.
} 
The condition~\eq{eq:simpcond} is the same as $\abs{\beta_2}<1$,
%
and the result now follows after noting 
$\frac{1}{2}\Delta_{ij} t(H_2, Z)\eqlaw B_1+B_2$, where $B_1,B_2\sim\Bi(n-2,\astar)$ and are independent, and so
$\Var(\Delta_{ij} t(H_2, Z))=8 (n-2)\astar (1-\astar)$. 
\end{proof}

\begin{proof}[Proof of Proposition~\ref{prop:tri}]
 Let $k=2$ and $H_2$ be a triangle. Applying Theorem~\ref{thm:negbetas}, we compute
\be{
\abs{\Phi}(a)= \abs{\beta_1} +3 \abs{\beta_2} a^2, \hspace{5mm} \abs{\Phi}'(a)=6 \abs{\beta_2} a,
}
and $\abs{\Phi}'(1)=
6 \abs{\beta_2}$. Then we can apply Theorem~\ref{thm:negbetas} 
assuming~\eq{eq:simpcond}, which is the same as $\abs{\beta_2}<1/3$.
The result now follows after noting that  $\frac{1}{6}\Delta_{ij} t(H_2, Z)\sim\Bi(n-2,(\astar)^2)$ and so 
$\Var(\Delta_{ij} t(H_2, Z))=36 (n-2)(\astar)^2 (1-(\astar)^2)$. 
\end{proof}

\begin{proof}[Proof of Theorem~\ref{thm:hightemp}]
As \cite{Bhamidi2011} showed, the $\ergm$ chain may not be contracting over the entire state space in the high temperature regime. Thus we will not use Theorem~\ref{thm:key}, but instead combine and apply Lemmas~\ref{lem:genapp1} and~\ref{lem:delf} directly.

For each $m\geq0$, let $\bclr{U\t{x,ij}(m),V\t{x,ij}(m)}$ be a coupling of the Glauber dynamics of the~$\ergm(\beta)$, with initial states $U\t{x,ij}(0)=x\s{ij,1}$, $V\t{x,ij}(0)=x\s{ij,0}$. Lemma~\ref{lem:delf} implies that
\ban{
\abs{\Delta_{ij} f_h(x)} &\leq 
\norm{\Delta h} \mathop{\sum_{st\in\gset{n}}}_{m\geq0} \IP\bclr{U\t{x,ij}_{st}(m)\not=V\t{x,ij}_{st}(m)} \notag   \\
	&\leq \norm{\Delta h} 
		\bbcls{\sum_{m<t_n} \IE \dham\bclr{U\t{x,ij}(m), V\t{x,ij}(m)}  \label{eq:t1}\\	
		& \qquad\qquad +\binom{n}{2}\sum_{m\geq t_n}^{\infty}\IP\bclr{U\t{x,ij}(m)\not=V\t{x,ij}(m)}}, \label{eq:t2}
}
where we set $t_n:=c n^2 \log(n)^2$, with $c$ to be chosen large enough (for~\eq{eq:t2tail} below). The main idea of the proof is to bound~\eq{eq:t2} by using that  the chain mixes in $n^2 \log(n)$ steps, and to bound~\eq{eq:t1} separately for the two cases where the $x$ is such that the chain stays in a contracting region for $t_n$ steps with high probability, and where $x$ is not. Then we show the former case occurs with high probability under the \ER\ distribution.

To bound~\eq{eq:t2}, we use \cite[Theorem~5]{Bhamidi2011}, which states that the $1/4$-mixing time $t_{\rm mix}$ of these Glauber dynamics is of order $t_{\rm mix}=\bigo(n^2 \log(n))$ and thus, using \cite[Exercise~4.3 and Formula~(4.33)]{Levin2009},   there are couplings for $m\geq t_n$,  such that for some $d>0$, 
\ba{
\IP\bclr{U\t{x,ij}(m)\not=V\t{x,ij}(m)}&\leq \IP\bbclr{U\t{x,ij}\bclr{\bfloor{\tsfrac{m}{t_{\rm mix}}} t_{\rm mix}}\not=V\t{x,ij}\bclr{\bfloor{\tsfrac{m}{t_{\rm mix}}} t_{\rm mix}}} \\
				&\leq 2^{-\bfloor{\tsfrac{m}{t_{\rm mix}}}} \\
				&\leq \bbclr{ e^{-\frac{d}{n^2\log(n)}}}^{m}.
}
We now easily find that for a constant $d'$ not depending on~$c$,
\ben{\label{eq:t2tail}
\binom{n}{2}\sum_{m\geq t_n}^{\infty}\IP\bclr{U\t{x,ij}(m)\not=V\t{x,ij}(m)}\leq d' n^4 \log(n) n^{-dc}.
}

Now, assume $n$ is large enough so that there is an $\eps>0$ such that~\eq{eq:contractbd} is less than one; 
call the right-hand side of \eq{eq:contractbd}~$\alpha<1$. This is always possible since~\eq{eq:contractbd} decreases as both $\eps\to0$ and $n\to\infty$, and has limit $\phi'(\astar)$, which is strictly less than one by assumption. For$\nu>0$, let $\cE_{\nu}$ be the set of $x\in\{0,1\}^{\binom{n}{2}}$ such that for all $ij\in\gset{n}$ and $H$ as in Lemma~\ref{lem:gdsetbd},~\eq{eq:subctbd} with $\eps$ replaced by $\nu$ is satisfied.
We only consider $\nu\leq \eps$, so that the chain is contracting while in $\cE_\nu$.

Let $0\leq m < t_n$, and $\bclr{U\t{x,ij}(m),V\t{x,ij}(m)}$ be the coupling from \cite{Bhamidi2011}. Denote the event that the coupling (and hence the path coupling by monotonicity) stays in the region $\cE_\nu$ up to time $t_n$ by
\be{
B_\nu:=\bigcap_{m<t_n} \bbclc{U\t{x,ij}(m)\in \cE_{\nu}, V\t{x,ij}(m)\in \cE_{\nu}}.
}
Combining~(18) of \cite{Bhamidi2011} and the comment before Claim~15 of \cite{Bhamidi2011}, as well as the comment before Lemma~16 of \cite{Bhamidi2011}, we find that for any $\mu>0, \delta>0$ small enough, if $x\in\cE_{2\mu}\setminus \cE_{\mu}$, then  there is a $d''>0$ such that
\ben{\label{eq:cgs}
\IP(B_{2\mu+\delta}^c ) \leq  t_n^4 e^{-d'' n}.
}
If $x\in \cE_{\mu}$, then in order for either chain to leave $\cE_{2\mu+\delta}$ it must first enter $\cE_{2\mu}\setminus \cE_{\mu}$, and so the Markov property and the argument above implies that~\eq{eq:cgs} holds also for $x\in \cE_{\mu}$. 
Now, choosing $\mu,\delta$ so that $2\mu+\delta<\eps$ and, in particular,
$B_\eps^c\subseteq B_{2\mu+\delta}^c$, we find that for $m<t_n$, 
\bes{
\IE \dham\bclr{U\t{x,ij}(m), V\t{x,ij}(m)}&\leq \II[x\in\cE_{\mu}]\IP[B_\eps]\left(1-{(1-\alpha)} \binom{n}{2}^{-1}\right)^m \\
&\qquad\quad+\bbclr{\II[x\not\in\cE_{\mu}] + \IP[B_\eps^c]}\binom{n}{2}. 
}
Combining these last two displays with~\eq{eq:cgs},~\eq{eq:t2tail},~\eq{eq:t2},~\eq{eq:t1},  noting the chain is contracting under $B_{\eps}$, and choosing~$c$ large enough yields
\be{
\abs{\Delta_{ij} f_h(x)}
	\leq \norm{\Delta h}\binom{n}{2}\left( {\{ (1- \alpha)}^{-1} 
		+\lito(1) \} \II[x\in\cE_{\mu}] +t_n \II[x\not\in\cE_{\mu}]\right).
}
Now applying Lemma~\ref{lem:genapp1} and noting that $\babs{q_X(\cdot|y)- q_Y(\cdot|y)}\leq 1$, we have
\bes{
\babs{\IE h(X)&-\IE h(Z)} \\
	 &\leq  \norm{\Delta h}\bbbclc{\left({(1- \alpha)}^{-1}+\lito(1)\right)\sum_{ij\in\gset{n}}\IE \bbcls{{\II[Z\in\cE_{\mu}] }\babs{q_X(Z\s{ij,1}|Z)- q_Z(Z\s{ij,1}|Z)}} \\
	&\hspace{3cm}+ t_n \sum_{ij\in\gset{n}} \IP\bclr{Z\not\in\cE_{\mu}}}.
}
The first term (after dropping the indicator) is bounded the same as in the proof of Theorem~\ref{thm:smallbetas} and Corollary~\ref{cor:vhightemp}. For the second term, we show that for all $H$ under consideration and $st\not=ij$, both $r_H(Z\s{ij,1},st)$ 
and $r_H(Z\s{ij,0},st)$ are highly concentrated, and then the result follows from a union bound.

From the definition of $r_H$ and elementary considerations of the function of the function $[0,1]\ni w\mapsto w^{1/(e_H-1)}$, we have
\be{
\IP\bbclr{\babs{r_H(Z\s{ij,\cdot},st)-\astar}>\mu}\leq\IP\bbbclr{\bbabs{\frac{\Delta_{st} t(H,Z\s{ij,\cdot})}{2  e_H n(n-1)\cdots (n-v_H+3)}-(\astar)^{e_H-1}}>\mu^{e_H-1}}.
}
Since
\ba{
\frac{\Delta_{st}t(H,Z\s{ij,\cdot})}{2  e_H n(n-1)\cdots (n-v_H+3)}
	&=\frac{\Delta_{st}t(H,Z)}{2  e_H n(n-1)\cdots (n-v_H+3)}+\bigo(n^{-1}) \\
	&=\frac{\Delta_{st}t(H,Z)}{2  e_H (n-2)\cdots (n-v_H+1)}+\bigo(n^{-1}),
}
we instead bound 
\ben{\label{eq:bbbd}
\IP\bbbclr{\bbabs{\frac{\Delta_{st} t(H,Z)}{2  e_H (n-2)\cdots (n-v_H+1)}-(\astar)^{e_H-1}}>\mu'},
}
for sufficiently large $n$ and small $\mu'>0$. Using the argument of \cite[Equation~(5), p.6]{Rucinski1988}; see also Remark~\ref{rem:var}; we have that for fixed $\astar$, and {fixed} $\ell\geq1$,
\be{
\IE\bbcls{\bclr{\Delta_{st} t(H,Z)- \IE\Delta_{st} t(H,Z)}^{2\ell}}{=}  \bigo\bclr{ n^{\ell(2\abs{v(H)}-5)}}.
}
Now noting that 
\be{
\IE\Delta_{st} t(H,Z)=2  e_H (n-2)\cdots (n-v_H+1) (\astar)^{e_H-1},
}
we use Markov's inequality in~\eq{eq:bbbd} to find that for fixed $\mu'$,
\be{
\IP\bbbclr{\bbabs{\frac{\Delta_{st} t(H,Z)}{2  e_H (n-2)\cdots (n-v_H+1)}-(\astar)^{e_H-1}}>\mu'}= \bigo\bclr{n^{-\ell}}.
}
By choosing~$\ell$ large enough and using a union bound, we have that
$\IP\bclr{Z\not\in\cE_{\mu}}=\bigo\bclr{n^{-1/2}}$, as desired.
\end{proof}

\section*{Acknowledgments}

GR acknowledges support from ARC grant DP150101459, from the Alan
Turing Institute, and from the COST Action CA15109.
NR acknowledges support from ARC grant DP150101459.
We thank the referee for pointing out the connection to Dobrushin mixing conditions, among other comments that have greatly improved the paper. We also thank 
Arthur Sinulis and Ronan Eldan for helpful remarks.


\end{document}